\documentclass[12pt]{amsart}
\usepackage{amssymb}
\usepackage{amsmath}
\usepackage{amsthm}
\usepackage{graphicx}
\usepackage{hyperref}
\usepackage{amscd}
\usepackage{enumerate}
\usepackage{mathrsfs}
\usepackage[latin1]{inputenc}
\usepackage{todonotes}

 \DeclareMathOperator{\perm}{Sym}
 \DeclareMathOperator{\soc}{soc}
\DeclareMathOperator{\aut}{Aut} \DeclareMathOperator{\out}{Out}

 \DeclareMathOperator{\frat}{Frat}

\DeclareMathOperator{\sym}{Sym}

\DeclareMathOperator{\core}{Core}

 \DeclareMathOperator{\diag}{Diag}

\newtheorem{thm}{Theorem}%[section]
\newtheorem{cor}[thm]{Corollary}

 \newtheorem{lemma}[thm]{Lemma}
\newtheorem{prop}[thm]{Proposition} \newtheorem{rem}[thm]{Remark}
 
\newtheorem{question}[]{Question} %\newtheorem{con}[]{Conjecture}
\numberwithin{equation}{section}

\renewcommand{\footnote}{\endnote}
\newcommand{\ignore}[1]{}\makeglossary
\begin{document}

\title[The Minimum Generating Set
Problem]{The Minimum Generating Set
	Problem}
	
		\author[A.~Lucchini]{Andrea Lucchini}
	\address[Andrea Lucchini]{Universit\`a di Padova, Dipartimento di Matematica \lq\lq Tullio Levi-Civita\rq\rq}
	\email{lucchini@math.unipd.it}
	\author[D.~Thakkar]{Dhara Thakkar}
	\address[Dhara Thakkar]{Indian Institute of Technology Gandhinagar, Gandhinagar}
	\email{thakkar\_dhara@iitgn.ac.in}

 \begin{abstract}Let $G$ be  a finite group. In order to determine the smallest cardinality $d(G)$ of a generating set of $G$ and a generating set with this cardinality, one should repeat \lq many times\rq \ the test whether a subset of $G$ of \lq small\rq \ cardinality generates $G$. We prove that if a chief series of $G$ is known, then the numbers of these \lq generating tests\rq \ can be drastically reduced. At most $|G|^{13/5}$ subsets must be tested. This implies that the minimum generating set problem for a finite group $G$ can be solved in polynomial time. \end{abstract}

\maketitle

\section{Introduction} Let $G$ be a finite group. A generating set of $G$ with minimum size is called a
minimum generating set. The size of a minimum generating set of a group $G$ is
denoted by $d(G).$ In this paper, we consider the problem of computing $d(G)$, known as the minimum generating set problem, and computing a minimum generating set of a given group $G$. We call \lq\lq a generating test\rq\rq \ for $G$, the check whether a given subset of $G$ generates $G.$ The computational cost of a single generating test for a given group $G$ may depends of the size of $G$ and of the subset that we are testing, and also of the way in which $G$ is assigned ($G$ could be a represented by permutation group representation, a multiplication table, a linear group representation, or an abstract group assigned with a presentation), but in any case an algorithm aimed to determine a minimum generating set becomes more efficient the smaller number of generating tests it requires. This motivates the following question.
\begin{question}\label{question-1}
How many generating tests are needed to determine a minimum generating set of a finite group?
\end{question}

Question \ref{question-1} has been studied before, mainly for groups represented by their multiplication table or by their permutation representation. With these representations of groups, it is easy to see that the minimum generating set problem is in the complexity class $\mathrm{NP}$, as there is a polynomial time algorithm to check if a given subset generates the given group \cite{seress}. Arvind and Tor\'an prove that the minimum generating set problem is in $\mathrm{DSPACE}(\log^2 n)$ \cite{AT}. Recently, the problem has been studied for general and special classes of groups by Das and Thakkar \cite{DT}. However, it is still open if the minimum generating set problem admits a polynomial time algorithm or equivalently if Question \ref{question-1} could be solved using at most polynomial many generating tests even for groups given by their multiplication tables. Information on the complexity classes $\mathrm{NP}$ and $\mathrm{DSPACE}(\log^2 n)$ can be found for example in \cite{Complexity-text}.

Notice that, for every finite group $G$, $d(G)\leq \log_p|G|$ being $p$ the smallest prime factor of $|G|$, so at least one set $\{g_1,\dots,g_d\}$ with $d\leq \log_p(|G|)$ and $1\neq g_i \in G$ for
$1\leq i\leq d$ is a minimum generating set. Hence a crude bound for the number of generating tests that are needed is $\sum_{t\leq \log_p(|G|)}(|G|-1)^t.$

A better strategy is suggested by the following  nice  remark  due to Gasch\"utz \cite {gaz1}.
	\begin{lemma}\label{gaz} Let $N$ be a normal subgroup of $G$ and let $g_1,
		\dots, g_d \in G$ be such that $G/N=\langle g_1N, \dots, g_dN \rangle$. If $G$
		can be generated with $d$ elements then there exist $u_1, \dots, u_d \in N$ such
		that $G=\langle g_1u_1, \dots, g_du_d \rangle.$ \end{lemma}
Suppose indeed that a chief series 

%\todo{I think there are typo in the below paragraphs. Please check the blue part I updated. Please ignore the changes if it is not correct}\todo{The problem with your correction is that it is not clear what you are saying when $i=u-1.$ In that case a minimum generating set for $G$ modulo $N_{i+1}=N_u=G$ is just the empty set, so $d=0$, so it is false that $d(G/N_i)\leq d+1=1$}
		$$1=N_0<N_1<\dots N_{u-1} <N_u=G$$
		of $G$ is available. The factor $G/N_{u-1}$ is a simple group, in particular $d(G/N_{u-1})\leq 2$, so at most $|G/N_{u-1}|^2$ generating tests are needed to find a subset $Y$ of minimum size with respect to the property that $G=\langle Y\rangle N_{u-1}.$ This can be considered a first step of an algorithm. For $0<i<u$ suppose that $\{g_1,\dots,g_d\}$ is a minimum generating set for $G$ modulo $N_i.$ By the main result in \cite{alold},
		$d(G/N_{i-1})\leq d+1.$ If  $d(G/N_{i-1})=d$, then by Lemma \ref{gaz} there exist $n_1,\dots,n_d \in N_i$ such that $\{g_1n_1,\dots,g_dn_d\}$ is a minimum generating set for $G$ modulo $N_{i-1}.$ If  $d(G/N_{i-1})=d+1$, then again by Lemma \ref{gaz} there exists $n_1,\dots,n_d,n_{d+1} \in N_i$ such that $\{g_1n_1,\dots,g_dn_d,n_{d+1}\}$ is a minimum generating set for $G$ modulo $N_{i-1}.$ In other words, if we know a minimum generating set for $G$ modulo $N_i$, then at most $|N_i/N_{i-1}|^d+|N_i/N_{i-1}|^{d+1}$ generating tests are needed to determine a minimum generating set for $G$ modulo $N_{i-1}.$ This approach reduces the number of generating tests, but this number remains in general quite large, and it does not give any evidence that a polynomial bound in terms of the order of $G$ is possible.
	Our main contribution is to show that the number of generating tests needed to obtained a minimum generating set of $G/N_{i-1}$ from a minimum generating set of $G/N_{i}$ can be drastically reduced (see in particular Corollaries \ref{coroab} and \ref{corononabb}). As a consequence, we obtain we following unexpected result.

\begin{thm}\label{main}The number of generating tests needed to determine a minimum generating set of a finite group $G$ is at most $|G|^{\frac{13}{5}}$.
\end{thm}

Our procedure requires the knowledge of a chief series of the group, that in any case can be computed in polynomial time (see for example \cite[6.2.7]{seress}), so we may conclude that the problem of finding a minimum generating set for a finite group $G$ has a solution which is polynomial on $|G|$.
	
Our proof of Theorem \ref{main} uses the theory of crowns of finite groups (see Section \ref{crowns}) and some results about the generation of finite groups with a unique minimal normal subgroup that strongly depend on the classification of the finite non-abelian simple groups.

%Next we study the minimum generating set problem for the class $\Gamma(c)$ of permutation groups, all of whose non-abelian composition factors are isomorphic to subgroups of $S_{c}$. The class $\Gamma(c)$ studied by Babai, Cameron, and P\'alfy (see for example \cite{bcp}). For groups in the class $\Gamma(c)$ we obtain the following result. 

\section{Chief series and chief factor}

Let $G$ be a nontrivial finite group. Recall that
a chief series of a finite group $G$ is a normal series
$$1=N_0 < N_1 < \dots < N_u = G$$
of finite length with the property that for $i\in \{0,\dots, u-1\}$, $N_{i+1}/N_i$ is a minimal normal subgroup of $G/N_i$. The integer $u$ is called the length of the series and the factors  $N_{i+1}/N_i$, where  $0\leq i \leq u-1$,   are called the chief factors of the series.
A nontrivial finite group $G$ always possesses a chief series. Moreover, two chief series of $G$ have the same length, and any two chief series of $G$ are the same up to permutation and isomorphism. Thus, adopting the notation above, we may define {the chief length} of $G$ to be $u$, and the chief factors of $G$ to be the groups $N_{i+1}/N_i$. 

In the following we will denote by $\frat(G)$ the Frattini subgroup of a finite group $G$. Moreover if $H/K$ is a chief factor of $G$, we say that $H/K$ is {Frattini} if $H/K\leq \frat(G/K)$ and that $H/K$ is complemented if there exists a subgroup $U$ of $G$ such that $UH=G$ and $U\cap H=K$.  Since the Frattini subgroup of a finite group $G$ is nilpotent, and the only subgroup supplementing $\frat(G)$ is $G$ itself, the following lemma is immediate.
\begin{lemma}\label{l:1}
	Let $G$ be a nontrivial finite group and let $A=H/K$ be a chief factor of $G$. 
	\begin{enumerate}[(i)]
		\item If $A$ is  abelian then $A$ is non-Frattini if and only if $A$ is complemented.  
		\item If $A$ is non-abelian then $A$ is non-Frattini. 
	\end{enumerate}
\end{lemma}

\section{Monolithic group and crown-based power}

A finite group $L$ is called {monolithic} if $L$ has a unique minimal normal subgroup $A$. If in addition $A$ is not contained in $\frat(L)$, then $L$ is called a {monolithic primitive group}. 

Let $L$ be a monolithic primitive group and let $A$ be its unique minimal normal subgroup. For each positive integer $k$,
let $L^k$ be the $k$-fold direct product of $L$. The crown-based power of $L$ of size  $k$ is the subgroup $L_k$ of $L^k$ defined by
$$L_k=\{(l_1, \ldots , l_k) \in L^k  \text{ : } l_1 \equiv \cdots \equiv l_k \ {\mbox{mod}\text{ } } A \}.$$

Equivalently, $L_k=A^k \diag(L^k)$, where $$\diag(L^k):=\{(l,l,\hdots,l)\text{ : }l\in L\}\le L^k.$$  

We also define $L_0:=1$. Assume that $A=\soc L$ is non-abelian and let $C:=C_{\aut A}(L/A).$  Moreover assume that $L=\langle l_1,\dots,l_d\rangle A.$ Suppose $d\geq d(L)$ and let $\Omega_{l_1,\dots,l_d}$ be the set
of $d$-tuples $(\overline l_1,\dots,\overline l_d)$ in $L^d$ such that $\langle 
\overline l_1,\dots,\overline l_d\rangle=L$ and $\overline l_i\equiv l_i \mod A$ for $1\leq i\leq d.$ Then $C$ acts on $\Omega_{l_1,\dots,l_d}$ 
by setting $(x_1,\dots,x_d)^\gamma=(x_1^\gamma,\dots,x_d^\gamma)$ for every $(x_1,\dots,x_d) \in  \Omega_{l_1,\dots,l_d}$  and $\gamma\in C.$
The following holds
(see \cite[Section 2]{austr}).
\begin{lemma}\label{comegenero}
	Assume $a_{ij}\in A,$ with $1\leq i\leq d$ and $1\leq j\leq k$ and let  $$\begin{aligned}
	g_1=&l_1(a_{11},\dots,a_{1k}),\\
	&\dots \quad \dots\quad\dots \quad\\
	g_d=&l_d(a_{d1},\dots,a_{dk}).
	\end{aligned}
	$$
	Then $\langle g_1,\dots,g_d\rangle=L_k$ if and only if $\langle l_1,\dots,l_d\rangle=L$ and 
	the $d$-tuples
	$(l_1a_{11},\dots,l_da_{d1}),\dots,$ $(l_1a_{1k},\dots,l_da_{dk})$ belong to different orbits for the action of $C$ on  $\Omega_{l_1,\dots,l_d}.$
\end{lemma}

Before to state the next result, we need to recall another important observation.  Let $H$ be a finite group, $X$  a subset of $H$, $M$ a normal subgroup of $H$ and assume that $h_1,\ldots,h_k\in H$ and $X$ generate $H$ modulo $M$, that is $H=\langle h_1,\ldots,h_k, X,M\rangle$.
It follows from \cite[Proposition 16]{Lxdir} that the number $\Phi_{H,M}(X,k)$ of elements $(u_1,\ldots, u_k)\in M^k$ with the property that $H=\langle h_1u_1,\ldots, h_ku_k,X\rangle$ is independent of the choice of $h_1,\ldots,h_k$.

\begin{lemma}\label{counting}
Let $L$ be a monolithic primitive group, and assume that $A=\soc(L)$ is non-abelian. Moreover, suppose that $d\geq d(L)$ and $L=\langle l_1,\dots,l_d\rangle A,$ and let $t$ be a positive integer with $2\leq t\leq d.$ Let $\Delta$ be the set of the $t$-tuples $(a_1,\dots,a_t)\in A^t$ with the property that $L=\langle l_1a_1,\dots,l_ta_t,l_{t+1},\dots,l_d\rangle.$ Then $$|\Delta|\geq\frac{53|A|^t}{90}.$$ 
\end{lemma}
\begin{proof}
Set $X=\{l_{t+1},\dots,l_d\}.$
It follows from the proof of \cite[Lemma 2]{index}, that there exist
$y_1,\dots,y_t\in L$ such that $L=\langle y_1,\dots,y_t,X\rangle.$ Hence $|\Delta|=\phi_{L,A}(X,t).$ 
%\todo{Do we need to use the proof of Lemma 2 \cite{index} to show $|\Delta|=\Phi_{L,A}(X,t)$ or it is clear from the def of $\Delta$ ? Answer: It follows from the definition; the proof of Lemma 2 is not needed}
To conclude it suffices to notice that, by \cite[Theorem 2.2]{london},  $\phi_{L,A}(X,t)\geq 53|A|^t/90.$
\end{proof}

\begin{prop}\label{smallt}Let $L$ be a monolithic primitive group, and assume that $A=\soc(L)$ is non-abelian. Consider the crown-based  product $G=L_\delta$ and let $N\cong A$ be a minimal normal subgroup of $G.$ Suppose that $G=\langle g_1,\dots,g_d\rangle N$, with $d\geq 2.$ Let $t=\lceil\frac{8}{5}+\log_{|N|}\delta\rceil$. If $t\leq d,$ then  there exist $n_1,\dots,n_t\in N$ such that $G=\langle g_1n_1,\dots,g_tn_t,g_{t+1},\dots,g_d\rangle.$
\end{prop}
\begin{proof}
Notice that $L/A\cong G/\soc(G),$ hence
$$d(L/A)=d(G/\soc(G))\leq d(G/N)\leq d.$$ It follows from the main theorem in \cite {unico}, that $$d(L)=\max\{2,d(L/N)\}\leq d.$$
	The elements of $G$ are of the form $l(y_1,\dots,y_\delta)$ with $l\in L$ and $y_1,\dots,y_\delta\in A.$ We may identify $N$ with the subgroup
	$\{(1,\dots,1,y)\mid y \in A\}$ of $A^\delta=\soc(G).$ Moreover, for $1\leq i\leq d$ and $1\leq j\leq \delta,$ there exist $l_i \in L$ and $y_{ij}$ in $A$ such that
	$$\begin{aligned}
	g_1=&l_1(y_{1,1},\dots,y_{1,\delta}),\\
	&\dots \quad \dots\quad\dots \quad\\
	g_d=&l_d(y_{d,1},\dots,y_{d,\delta}).
	\end{aligned}
	$$
	Let $\Gamma=C_{\aut A}(L/A)$ and, for $1\leq i\leq \delta-1,$ let
	$$\omega_i=(l_1y_{1,i},\dots,l_dy_{d,i})\in L^d.$$ Moreover let $\Omega$ be the set
	of  $(\overline l_1,\dots,\overline l_d)$ in $L^d$ such that $\langle 
	\overline l_1,\dots,\overline l_d\rangle=L$ and $\overline l_i\equiv l_i \mod A$ for $1\leq i\leq d.$ The condition $G=\langle g_1,\dots,g_d\rangle N$ implies that $\omega_i\in \Omega$ for $1\leq i\leq \delta-1$ and if $i\neq j$ then $\omega_i$ and $\omega_j$ belong to different $\Gamma$-orbits for the action of $\Gamma$ in $\Omega.$
	
	Our statement is equivalent to show that there exist $x_1,\dots,x_t \in A$ such that $\omega=(l_1 y_{1,\delta} x_1,\dots,l_t y_{t,\delta} x_t,l_{t+1}y_{t+1,\delta},\dots,l_{d}y_{d,\delta})\in \Omega$ and, for each $1\leq i\leq \delta-1$,  $\omega$ and $\omega_i$ belong to different $\Gamma$-orbits. Indeed in that case, taking, for $1\leq i\leq t,$ $n_i=(1,\dots,1,x_i)\in N,$ by Lemma \ref{comegenero} we conclude $G=\langle g_1n_1,\dots,g_tn_t,g_{t+1},\dots,g_d\rangle.$

Let $\Delta$ be the subset of $\Omega$ consisting of the $d$-tuples  $(\overline l_1,\dots,\overline l_d)$ with $\overline l_i =l_iy_{i\delta}$ for each $i>t.$ Assume $A\cong S^n,$ with $n\in \mathbb N$ and $S$ a finite non-abelian simple group. By Lemma \ref{counting},
$$|\Delta|\geq  \frac{53|A|^{t}}{90}.$$
Notice that the action of $\Gamma$ on $\Omega$ is semiregular. Since $$|\Gamma| \leq n|S|^{n-1}|\aut(S)|\leq n|A||\out(S)|\leq n|A|\log_2|S|
\leq |A|\log_2|A|$$ (see the proof of \cite[Lemma 1]{pr} and the upper bound for $|\out(S)|$ given in \cite{ku}), the number of
$\Gamma$-orbits with a representative in $\Delta$ is at least
$$\frac{|\Delta|}{|\Gamma|}\geq \frac{53}{90}\frac{|A|^{t-1}}{\log_2|A|}\geq |A|^{t-\frac{8}{5}}.$$ Since  $|A|^{t-8/5}=|N|^{t-8/5}\geq \delta$, we can conclude that $\Delta$ contains an element with the required property.
\end{proof}

\section{Crowns}\label{crowns}

The notion of crown was
introduced by Gasch\"{u}tz  in \cite{Gaschutz} in the case of finite solvable groups and generalized
in \cite{JL} to arbitrary finite groups.
A detailed exposition of the theory is also given in \cite[1.3]{classes}.

If a group $G$ acts on a group $A$ via automorphisms (that is, if there exists a homomorphism $G\rightarrow \aut(A)$), then we say that $A$ is a $G$-group. If $G$ does not stabilise any nontrivial proper subgroup of $A$, then $A$ is called an irreducible $G$-group. Two $G$-groups $A$ and $B$ are said to be $G$-isomorphic, or $A\cong_G B$, if there exists a group isomorphism $\phi: A\rightarrow B$ such that 
$\phi(g(a))=g(\phi(a))$ for all $a\in A, g\in G$.  Following  \cite{JL}, we say that two  $G$-groups $A$ and $B$  are $G$-equivalent and we put $A \equiv_G B$, if there are isomorphisms $\phi: A\rightarrow B$ and $\Phi: A\rtimes G \rightarrow B \rtimes G$ such that the following diagram commutes:

\begin{equation*}
	\begin{CD}
		1@>>>A@>>>A\rtimes G@>>>G@>>>1\\
		@. @VV{\phi}V @VV{\Phi}V @|\\
		1@>>>B@>>>B\rtimes G@>>>G@>>>1.
	\end{CD}
\end{equation*}

\

Note that two  $G$\nobreakdash-isomorphic $G$\nobreakdash-groups are $G$\nobreakdash-equivalent. In the particular case where $A$ and $B$ are abelian the converse is true: if $A$ and $B$ are abelian and $G$\nobreakdash-equivalent, then $A$ and $B$ are also $G$\nobreakdash-isomorphic.
It is proved (see for example \cite[Proposition 1.4]{JL}) that two  chief factors $A$ and $B$ of $G$ are $G$-equivalent if and only if either they are  $G$-isomorphic, or there exists a maximal subgroup $M$ of $G$ such that $G/\core_G(M)$ has two minimal normal subgroups $X$ and $Y$ that are
$G$-isomorphic to $A$ and $B$ respectively. For example, the minimal normal subgroups of a crown-based power $L_k$ are all $L_k$-equivalent.

For an irreducible $G$-group $A$ denote by $L_A$ the monolithic primitive group associated to $A$.
That is
$$L_{A}=
\begin{cases}
	A\rtimes (G/C_G(A)) & \text{ if $A$ is abelian}, \\
	G/C_G(A)& \text{ otherwise}.
\end{cases}
$$
If $A$ is a non-Frattini chief factor of $G$, then $L_A$ is a homomorphic image of $G$. More precisely, there exists a normal subgroup $N$ of $G$ such that $G/N \cong L_A$ and $\soc(G/N)\equiv_G A$. Consider now all the normal subgroups $N$ of $G$ with the property that $G/N \cong L_A$ and $\soc(G/N)\equiv_G A$:
the intersection $R_G(A)$ of all these subgroups has the property that  $G/R_G(A)$ is isomorphic to the crown-based power $(L_A)_{\delta_G(A)}$. The socle $I_G(A)/R_G(A)$ of $G/R_G(A)$ is called the {$A$-crown} of $G$ and it is  a direct product of $\delta_G(A)$ minimal normal subgroups $G$-equivalent to $A$.

In our proof we will use the following consequence of \cite[Proposition 1.1]{crown}.

\begin{lemma}\label{corone}
	Let $N$ be  a minimal normal subgroup of a finite group $G.$ If $HN=HR_G(N)=G,$ then $H=G.$
\end{lemma}

The next result is an immediate consequence of \cite[Lemma 10]{crown2}.
\begin{lemma}\label{corone2}
	If $N$ is a non-Frattini minimal normal subgroup of a finite group $G$, then 
	$NR_G(N)/R_G(N)$ is $G$-equivalent to $N.$
	\end{lemma}

\section{Abelian minimal normal subgroups}

\begin{prop}\cite[Theorem 4]{jsc}\label{jscc} 
	Let $G$ be a group and let $N = \langle e_1,\dots, e_l\rangle$ be an abelian
	minimal normal subgroup of $G.$ If $G= \langle g_1,\dots, g_d\rangle$N then one of the following occurs:
	\begin{enumerate}
		\item $d(G) \leq d$ and either $G = \langle g_1,\dots, g_d\rangle$ or there exist $1 \leq i \leq d$ and $1 \leq  j \leq l$ such that $G = \langle g_1,\dots, g_{i-1}, g_ie_j, g_{i+1},\dots, g_d\rangle;$ 
		\item $d(G) = d + 1$ and $G = \langle g_1, \dots , g_d, x\rangle$ for any $1\neq x \in N.$
	\end{enumerate}
\end{prop}

\begin{cor}\label{coroab}
	Let $N$ be a minimal abelian normal subgroup of a finite group $G$ and let $g_1N,\dots,g_dN$ be a minimum generating set of $G/N.$ We can find a minimal generating set for $G$ testing at most $d(|N|-1)+1\leq d|N|$  elements of $|G|^d.$  
\end{cor}
\begin{proof}
	By the previous proposition, if $d(G)=d,$ then $\{g_1,\dots,g_d\}$ or $\{g_1,\dots,g_in,\dots g_d\}$, with $1\neq n \in N$ and $1\leq i\leq d,$ is a minimal generating set of $G$. If none of these sets generates $G$, then $d(G)=d+1$ and, for every $1\neq n\in N,$ $\{g_1,\dots,g_d,n\}$ is a generating set.
\end{proof}

\begin{rem}\label{remab}
If $N$ is an abelian minimal normal subgroup of a finite group $G$, then $N \cong C_p^l$ is an elementary abelian $p$-group and a generating set $e_1,\dots,e_l$ for $N$ can be obtained in $l$ steps. So we may improve the statement of the previous corollary stating that a minimum generating set of $G$ can be obtained from a minimum generating set   of $G/N$ in at most $l+dl+1=(d+1)l+1$ steps.
\end{rem}

\section{Non-abelian minimal normal subgroups}

\begin{lemma}\label{mint}Let $N$ be a non-abelian minimal normal subgroup of a finite group $G.$  Suppose that $G=\langle g_1,\dots,g_d\rangle N,$ with $d\geq 2$. Let $t=\left\lceil\frac{8}{5}+\log_{|N|}\delta_G(N)\right\rceil.$ If $t\leq d,$ then there exist $n_1,\dots,n_t\in N$ such that $$G=\langle g_1n_1,\dots,g_tn_t,g_{t+1},\dots,g_d\rangle.$$
\end{lemma}

\begin{proof}
	Let $R=R_G(N)$ and $\delta=\delta_G(N).$  Moreover let $L$ be the monolithic primitive group associated to $N.$  We use the bar notation $\bar g$ to denote the elements of the factor group $\bar G=G/R$. By Lemma \ref{corone2}, $\bar N=NR/R$ is a minimal normal subgroup of $G/R$. Moreover it follows from Lemma \ref{corone} that $G=\langle g_1n_1,\dots,g_tn_t,g_{t+1},\dots,g_d\rangle$ if and only if $\bar G=\langle \bar g_1\bar n_1,\dots,\bar g_t\bar n_t,\bar g_{t+1},\dots\bar g_d\rangle$. Since $\bar G \cong L_\delta,$ with $L\cong G/C_G(N),$ the conclusion follows from Proposition \ref {smallt}
\end{proof}

\begin{cor}\label{corononab}
	Let $N$ be a minimal non-abelian normal subgroup of a finite group $G$ and let $g_1N,\dots,g_dN$ be a minimum generating set of $G/N.$ We may produce a minimum generating set for $G$ testing at most $|N|^{\lceil\frac{8}{5}+\log_{|N|}\delta_G(N)\rceil}$  elements of $G^d$ and at most $|N|^{\lceil\frac{8}{5}+\log_{|N|}\delta_G(N)\rceil}$ elements of $G^{d+1}.$
\end{cor}

\begin{proof}
	Let	$u=\max\{d,2\}$. If $d<u,$ set $g_{d+1}=\dots =g_u=1.$ Let $t=\min\{u,\lceil\frac{8}{5}+\log_{|N|}\delta\rceil\}.$ First we test the
	$|N|^t$ $u$-tuples of kind $(g_1n_1,\dots,g_tn_t,g_{t+1},\dots,g_u)$,
	with $n_1,\dots,n_t \in N.$ If one of them is a generating set, we are done. If not, it follows from Lemma \ref{mint} that $u<t$. In this case we have tested that $\langle g_1n_1,\dots,g_un_u\rangle \neq G$ for every $(n_1,\dots,n_u)\in N^u,$ so if follows from Lemma \ref{gaz} that $d(G)>u.$
 Moreover,
	by the main result in \cite{alold}, $$u<d(G)\leq \max\{2,d(G/N)+1\}\leq u+1\leq t$$ Hence $d(G)=u+1\leq t$ and, by Lemma \ref{gaz},
 one of the   $(u+1)$-tuples of kind $(g_1n_1,\dots,g_dn_u,n_{u+1})$ is a minimum generating set
\end{proof}

%\todo{I tried to understand above proof. I rewrite the proof from what I understand. Let me know if it is correct or some changes required!}

%\todo{I still don't see the reason of your second paragraph In my opinion is only makes confusion to the reader. What are your problems with my proof?}

%\todo{I do not understand how $d(G)=u+1$ in your proof? So I tried to write what I understand to make clear that I am understanding your proof correctly.}

In applying the previous corollary to design an algorithm to compute a minimum generating set of a finite group $G,$ an obstacle is given by the fact that,
even if a chief series of $G$ is available, it is not easy to recognize whether two factors of the series are $G$-equivalent. Hence, with the computation applications in mind, it is better to consider the weaker equivalence relation in which two factors are equivalent if they have the same order. Denoting with $\eta_G(A)$ the number of factors in a chief series of $G$ with order $|A|,$ we may state the weaker formulation of the previous corollary.

\begin{cor}\label{corononabb}
	Let $N$ be a minimal non-abelian normal subgroup of a finite group $G$. If we know a minimum generating set of $G/N$, then the number of generating tests needed to compute a minimum generating set for $G$ is at most $$2|N|^{\lceil\frac{8}{5}+\log_{|N|}\eta_G(N)\rceil}.$$
\end{cor}
\section{Finding a minimal generating set in polynomial time}

\begin{proof}[Proof of Theorem \ref{main}]
Let $\mathcal A$ and $\mathcal B$ be, respectively,  the set of abelian and non-abelian factors in a given chief series of $G.$ Moreover let $\alpha(G)=\prod_{X\in \mathcal A}|A|$ and $\beta(G)=\prod_{Y\in \mathcal B}|B|.$
With iterated application of Corollaries  \ref{coroab} and \ref{corononabb}, we deduced that the number of generating tests required to obtain a minimum generating set for $G$ is at most
$$\sum_{X \in \mathcal A}d(G)|X|+\sum_{Y \in \mathcal B}2|Y|^{\log_{|Y|}\eta_G(Y)+\frac{13}{5}}.$$
Clearly $$\sum_{X\in \mathcal A}d(G)|X|\leq d(G)\prod_{X\in \mathcal A}|X|=d(G)\alpha(G).$$
We consider the equivalence classes in $\mathcal B$ in which two factors are equivalent if and only if they have the same order. 
Let $\mathcal B_1,\dots,\mathcal B_r$ be the equivalence classes in $\mathcal B$ and for every class choose a representative $Y_i$ for this class. Then we have 
$$\begin{aligned}\sum_{Y \in \mathcal B}2|Y|^{\log_{|Y|}\eta_G(Y)+\frac{8}{3}}&\leq	\sum_{1\leq i \leq r}2\eta_G(Y_i)|Y_i|^{\log_{|Y_i|}\eta_G(Y_i)+\frac{13}{5}}\\
	&=\sum_{1\leq i \leq r}2\eta_G(Y_i)^2|Y_i|^{\frac{13}{5}}\\
		&\leq \sum_{1\leq i \leq r}|Y_i|^{\frac{13}{5} \eta_G(Y_i)}  \\   
		&\leq \left(\prod_{1\leq i\leq r}|Y_i|^{\eta_G(Y_i)}   
		\right)^{\frac{13}{5}}=\beta(G)^{\frac{13}{5}}.
	\end{aligned}
$$

If follows that the number of required generating tests is at most $$d(G)\alpha(G)+{\beta(G)^{\frac{13}{5}}}\leq |G|^{\frac{13}{5}}.\qedhere$$
\end{proof}

\section{Permutation groups}

The best bound for the cardinality of a generating set of a permutation group is due to A. McIver and P. Neumann: the so call \lq\lq McIver-Neumann Half-$n$ Bound\rq\rq\ says that if $G$ is a subgroup of $\perm(n)$ and $G\neq \perm(3),$ then $d(G)\leq \lfloor n/2 \rfloor.$ This result is stated without a proof in \cite[Lemma 5.2]{min} and a sketch of the proof is given in \cite[Section 4]{cst}. 
The following result is crucial for our purposes. 

\begin{thm}\cite[Theorem 10.0.5]{nina-thesis}\label{ninatesi}. Let $G$ be a permutation group of degree n with s orbits. Then a chief series of $G$ has length at most $n-s.$
	\end{thm}

\begin{lemma}\label{permab} Let $G\leq \sym(n)$ and
 $N$  a minimal abelian normal subgroup of  $G$. If we know a minimum generating set for $G/N$, we can find a minimum generating set for $G$ in at most $(\lfloor n/2\rfloor+1)^2$ steps.
\end{lemma}

\begin{proof}We may assume $n>3.$ Then
 $\max\{d(G), d(N)\} \leq  n/2$, so the conclusion follows from Remark \ref{remab}.
\end{proof}

\begin{lemma}\label{permnonab} Let $G\leq \sym(n)$
and	$N$  a minimum non-abelian normal subgroup of  $G$. If we know a minimum generating set for $G/N$, we can find a minimum generating set for $G$  with at most $(n-1)|N|^{\frac{13}{5}}$ generating tests.
\end{lemma}
\begin{proof}
By Corollary  \ref{corononab}, we need at most $|N|^{\log_{|N|}\delta_G(N)+\frac{13}{5}}=\delta_G(N)|N|^{\frac{13}{5}}$ generating tests. Moreover from Theorem \ref{ninatesi}, it follows $\delta_G(N)\leq n-1.$
\end{proof}

In the following we will denote by $\lambda(G)$ the maximum of the orders of the non-abelian chief factors of $G$ (and we will set $\lambda(G)=1$ if $G$ is solvable).	Combining Theorem \ref{ninatesi} and Lemmas \ref{permab} and \ref{permnonab}
we obtain:

\begin{thm}
Let $G\leq \sym(n).$ Then we may obtain a minimum generating set of $G$ with at most $n^2\lambda(G)^{13/5}$ generating test.
\end{thm}

\begin{cor}
Let $G\leq \sym(n).$ If every non-abelian composition factor of $G$ has order at most $c$, then  we may obtain a minimum generating set of $G$ with at most $n^2c^{\frac{13}{20}n}$ generating tests.
\end{cor}

\begin{proof}
Let $H/K$  be a non-abelian chief factor of $G.$ Then $H/K\cong S^u,$ with $u \in \mathbb N$ and $S$ a finite non-abelian simple group. Let $X/K$ be  a Sylow 2-subgroup of $H/K.$ We have $X/K \cong P^u$  with $P$ a Sylow 2-subgroup of $S$. Since $P$ cannot be cyclic, 
$2u\leq d(X/K) \leq d(X)\leq n/2,$  hence $u\leq n/4$  and $|H/K|\leq c^{n/4},$
so the conclusion follows from the previous Theorem.
\end{proof}

%\textcolor{red}{Note: In our preprint, we considered the case of Primitive Permutation group, a quasi-polynomial time algorithm (Theorem 16). Here we just find $d(G)$. Our initial goal was to find a generating set of size $d(G)$. The problem is that, we can find a generating set of size $d(G)$ trivially in all the case except product type (Theorem 15), where we can find a generating set of size $d(G/soc(G))(=d(G))$ of the quotient group $G/soc(G)$ (here $soc(G)$ is unique minimal). We do not know how to extend this to $d(G)$ size generating set for $G$. I think the proof of Theorem 1.1 in [Andrea Lucchini and Federico Menegazzo. Generators for finite groups with a unique minimal normal subgroup, 1997] will help to return one $d(G)$ size generating set for $G$.}\textcolor{blue}{No hope for this!}


\begin{thebibliography}{99}

\bibitem{Complexity-text}
S. Arora and B. Barak, Computational complexity: a modern approach. Cambridge University Press (2009).
	
\bibitem{AT}
V. Arvind, and J. Tor\'{a}n, The complexity of quasigroup isomorphism and the minimum generating set problem, ISAAC 2006, Proceedings 17 (2006), 233--242.
%\bibitem{bcp} L. Babai, P. Cameron, and P. P\'alfy, On the orders of primitive groups with restricted nonabelian composition factors, Journal of Algebra 79, no. 1 (1982), 161-168.
\bibitem{classes}
A. Ballester-Bolinches and L.~M. Ezquerro, Classes of finite
	groups, Mathematics and Its Applications (Springer), vol. 584, Springer,
Dordrecht (2006).
\bibitem{cst} P. Cameron, R. Solomon and A. Turull, Chains of subgroups in symmetric groups, J. Algebra 127 (1989), no. 2, 340--352.
\bibitem{austr} F. Dalla Volta and A. Lucchini, Finite groups that need more generators than any proper quotient, J. Austral. Math. Soc. Ser. A 64 (1998), no. 1, 82--91.
\bibitem{pr} F. Dalla Volta and A. Lucchini, The smallest group with nonzero presentation rank, J. Group Theory 2 (1999), 147--155.
\bibitem{london} E. Detomi and A. Lucchini, Probabilistic generation of finite groups with a unique minimal normal subgroup, J. Lond. Math. Soc. (2) 87 (2013), no. 3, 689--706.	
\bibitem{crown} E. Detomi and A. Lucchini,  Crowns and factorization of the probabilistic zeta function of a finite group, J. Algebra 265 (2003), no. 2, 651--668.
\bibitem{DT}
B. Das and D. Thakkar, Algorithms for the Minimum Generating Set Problem, arXiv preprint arXiv:2305.08405 (2023).
\bibitem{crown2} E. Detomi and A. Lucchini,  Crowns in profinite groups and applications, Noncommutative algebra and geometry, 47--62, Lect. Notes Pure Appl. Math., 243, Chapman \& Hall/CRC, Boca Raton, FL, 2006. 
\bibitem{gaz1} W. Gasch\"utz, Zu einem von B. H. Neumann und H. Neumann gestellten Problem, Math. Nachr. 14 (1955), 249-252.
\bibitem{Gaschutz} W. Gasch\"{u}tz, Praefrattinigruppen, Arch. Math. {13} (1962), 418--426.
\bibitem{JL} P. Jim\'{e}nez-Seral and J. Lafuente, On complemented non-abelian chief factors of a finite group, Israel J. Math. {106} (1998), 177--188. 
\bibitem{ku} S. Kohl, A bound on the order of the outer automorphism group of a
finite simple group of given order, online note (2003), available under
\texttt{http://www.gap-system.org/DevelopersPages/StefanKohl/preprints/ outbound.pdf}.
\bibitem{alold} A. Lucchini, Generators and minimal normal subgroups, Arch. Math. (Basel) 64 (1995), no. 4, 273--276.
\bibitem{Lxdir} A. Lucchini, The $X$-Dirichlet polynomial of a finite group,  J. Group Theory, 8 (2005), 171--178.
\bibitem{unico} A. Lucchini and F. Menegazzo, 
Generators for finite groups with a unique minimal normal subgroup,
Rend. Sem. Mat. Univ. Padova 98 (1997), 173?191.
\bibitem{jsc} A. Lucchini and F. Menegazzo, Computing a set of generators of minimal cardinality in a solvable group, J. Symbolic Comput. 17 (1994), no. 5, 409--420.
\bibitem{index} A. Lucchini and M. Morigi,  Recognizing the prime divisors of the index of a proper subgroup, J. Algebra 337 (2011), 335--344. 
\bibitem{min} A. McIver, P. Neumann,
Enumerating finite groups, Quart. J. Math. Oxford Ser. (2) 38 (1987),
no. 152, 473-488.
\bibitem{nina-thesis} N. Menezes, Random generation and chief length of finite groups, PhD Thesis, http://research-repository.st-andrews.ac.uk/handle/10023/3578.
%\bibitem{nina}  N. Menezes, M. Quick and C. Roney-Dougal, The probability of generating a finite simple group, Israel J. Math. 198 (2013), no. 1, 371--392.
\bibitem{seress} A. Seress, Permutation group algorithms, volume 152. Cambridge University
Press, 2003.

\end{thebibliography}
\end{document}